\documentclass[a4paper,10pt]{article}
\usepackage{amsmath, amsthm, amsfonts, amssymb}
\usepackage{mathrsfs}
\usepackage{mathtools}
\usepackage{hyperref}

\usepackage{xcolor}
\usepackage{bbm}

\newcommand{\be}{\begin{equation}}
	\newcommand{\ee}{\end{equation}}

\newcommand{\ba}{\begin{equation}\begin{aligned}}
		\newcommand{\ea}{\end{aligned}\end{equation}}

\newcommand{\di}{\mathrm{d}}

\newcommand{\ve}{\varepsilon}

\newcommand{\E}{{\mathop{\rm E}}}

\newcommand{\Levy}{L\'{e}vy }

\theoremstyle{plain}
\newtheorem{theorem}{Theorem}[section]
\newtheorem{lemma}[theorem]{Lemma}
\newtheorem{proposition}[theorem]{Proposition}
\newtheorem{corollary}[theorem]{Corollary}

\theoremstyle{definition}

\numberwithin{equation}{section}

\let\oldmarginpar\marginpar
\renewcommand{\marginpar}[1]{\oldmarginpar{\scriptsize\texttt{\color{red}{#1}}}}

\begin{document}

\title{Functional central limit theorem for superdiffusive SDEs  with stable noise}

\author{Aleksandar Mijatovi\'c\footnote{University of Warwick, UK (\texttt{a.mijatovic@warwick.ac.uk})},
	\quad Andrey Pilipenko\footnote{  Universit\'e de Gen\`eve, 
Section de math\'ematiques, Switzerland (\texttt{andrey.pilipenko@unige.ch})} \footnote{ Institute of Mathematics of the National Academy of Sciences of Ukraine}  \quad  and \quad 
	Isao Sauzedde\footnote{École normale supérieure de Lyon, France (\texttt{isao.sauzedde@ens-lyon.fr})}
}

\maketitle

\begin{abstract}
This paper establishes a functional stable central limit theorem for a class of superdiffusive solutions to stochastic differential equations (SDEs) driven by an $\alpha$-stable process.  
\end{abstract}


\smallskip


\section{The main result}
\label{sec:Introduction}

Let $B_\alpha$ be a strictly $\alpha$-stable process, $\alpha\in(0,2)\setminus\{1\}$, with Lévy measure $(c_+\mathbbm{1}_{u>0}+c_- \mathbbm{1}_{u<0})  |u|^{-1-\alpha} \di u$, where $c_+,c_-\geq0$ and $c_++c_->0$.
Let a smooth function $f:\mathbb{R}\to\mathbb{R}$ satisfy 
$f(x)\sim a x^\beta$ as  $x\to+\infty$ (i.e. $f(x)/(ax^\beta)\to1$ as $x\to+\infty$) for some $a>0$ and $\beta\in\mathbb{R}$.
Consider the stochastic differential equation~(SDE)
\begin{equation}
\label{eq:sde}
dX(t) =f(X(t))dt + dB_\alpha(t),\quad t\geq 0,
\end{equation}
and assume 
\begin{equation}\label{eq:assumption}
\lim_{t\to\infty}X(t)=+\infty \ \ \text{almost surely.}
\end{equation}
Then, by~\cite[Thms~2.1 \& 2.2]{pilipenko2021small},
the following almost sure limit holds for 
any		$\beta\in(1-\alpha,1)$: 
\begin{equation}
\label{eq:asympt_power}
X(t)\sim c_1 t^{\frac{1}{1-\beta}} , \qquad c_1\coloneqq (a(1-\beta))^{\frac{1}{1-\beta}} \ \ \text{a.s. as $t\to \infty$.}	
\end{equation}

Equation~\eqref{eq:sde} may be viewed as a nonlinear random
perturbation of the deterministic equation $\dot{x}=f(x)$. When
$f(x)\sim ax^\beta$ with $\beta<1$, the corresponding deterministic
trajectory grows on the scale $t^{1/(1-\beta)}$. By contrast, the
natural scale of the strictly $\alpha$-stable noise is
$t^{1/\alpha}$. Strictly stable L\'evy processes are the canonical
non-Gaussian limits in the generalised central limit theorem and
provide basic models of self-similar, heavy-tailed jump fluctuations;
see, for example,
\cite{GnedenkoKolmogorov,Samorodnitsky+Taqqu:1994,Sato}. Thus,
SDE~\eqref{eq:sde} provides a natural setting in which nonlinear
deterministic growth competes with fluctuations generated by jumps
of infinite variance. The inequality $\beta>1-\alpha$ is precisely
the condition that the deterministic growth exponent
$1/(1-\beta)$ is larger than the noise exponent $1/\alpha$, and the
strong law~\eqref{eq:asympt_power} shows that, in this regime, the
random perturbation does not alter the leading-order behaviour. The
natural next question is whether the deviations from the
deterministic trajectory possess a universal scaling limit.

Long-time asymptotics for equations of the form~\eqref{eq:sde} are
also closely related, through a space--time rescaling, to small-noise
perturbations of ordinary differential equations with non-Lipschitz
drift and to the associated Peano selection problem; see
\cite{PavlyukevichPilipenko2020,pilipenko2021small}. In that setting,
the selection of a solution near a non-unique equilibrium is governed
by the long-time behaviour of an SDE driven by self-similar noise.
Theorem~\ref{th:main1} addresses the corresponding second-order
question: after the deterministic growth has been removed, it
identifies the scale and the functional form of the remaining
fluctuations.

\begin{theorem}
\label{th:main1}
    Let $\beta\in(1-\alpha, \frac{1}{1+\alpha})$. Assume $f$ is a $\mathcal{C}^1$ function such that $(\ln(f))'(x)=O(1/x)$ and  $f(x)=a x^\beta +o(x^{\beta-r})$,
with $r=(\alpha+\beta-1)/\alpha$, as $x\to \infty$.
    If in addition~\eqref{eq:assumption} holds,
    then, as $\lambda\to \infty$, 
    \[X^{(\lambda)}\coloneqq \left( \frac{X(\lambda t)-c_1 ( \lambda t)^\frac{1}{1-\beta}}{(\lambda/\rho)^{1/\alpha}}\right)_{t\in[0,\infty)}   \overset{(d)}{\implies} (  t^\gamma B_\alpha(t^\rho))_{t\in[0,\infty)}, \quad \text{where } \gamma\coloneqq\frac{\beta}{1-\beta} \text{ $\&$ }  \rho\coloneqq 1- \frac{\alpha\beta}{1-\beta}, \] 
    and the convergence is in distribution in the space of càdlàg functions (endowed with the topology of local uniform convergence on $(0,\infty)$).  
    
    If $\beta\geq 0$, the convergence in distribution of $X^{(\lambda)}$ holds locally uniformly on $[0,\infty)$. If  $\beta<0$, then $(t^{-\gamma} X^{(\lambda)}(t))_{t \in[0,\infty)}$ converges in distribution to $ (B_\alpha(t^\rho))_{t\in[0,\infty)} $ locally uniformly on $[0,\infty)$.
\end{theorem}


The two constraints on $\beta$ in Theorem~\ref{th:main1} correspond
to two different transitions. The lower threshold
$\beta=1-\alpha$ determines whether the deterministic trajectory
dominates the noise at first order. The origin of the upper threshold
can be seen by a formal linearisation. In the model case
$f(x)=ax^\beta$, let
$q(t)=c_1t^{1/(1-\beta)}$ (with $c_1$ in~\eqref{eq:asympt_power}) be the deterministic trajectory and $Z(t)$ the random fluctuations of $X(t)$ around $q(t)$. Differently put, we have 
$X(t)=q(t)+Z(t)$, implying 
$\di Z(t)
= (f(X(t))-q'(t))\di t + \di B_\alpha(t)$
and
$f(X(t))\approx f(q(t))+f'(q(t)) Z(t)$.
Since $q'(t)=f(q(t))$ and 
\[
 f'(q(t))=\frac{\beta}{1-\beta}\frac{1}{t}
         =\frac{\gamma}{t},
 \qquad \gamma=\frac{\beta}{1-\beta},
\]
the linearised fluctuation $Z(t)$ formally satisfies 
\[
 \mathrm{d}Z(t)\approx \frac{\gamma}{t}Z(t)\,\mathrm{d}t
                   +\mathrm{d}B_\alpha(t),
 \quad\text{ and hence }\quad
 Z(t)\approx t^\gamma\int_0^t s^{-\gamma}\,
                         \mathrm{d}B_\alpha(s).
\]
The stable stochastic integral on the right is well defined at the
origin precisely when
$\int_0^1s^{-\alpha\gamma}\,\mathrm{d}s<\infty$,
or equivalently when $\alpha\gamma<1$, which is exactly
$\beta<1/(1+\alpha)$. By stable self-similarity, this integral gives
the time-changed process appearing in Theorem~\ref{th:main1}. The
theorem therefore makes the linearisation heuristic rigorous at the
functional level, while Proposition~\ref{le:counter} below shows that above
the upper threshold even the one-dimensional normalised fluctuations
cannot converge. In particular, there is a parameter regime in which
the deterministic strong law remains valid but no stable
second-order limit exists.

Note that the condition on the drift $f$ in terms of $r$ in Theorem~\ref{th:main1} is also sharp: it  ensures that the random fluctuations (of the order of magnitude $t^{1/\alpha}$) are not overwhelmed by the deterministic fluctuations of a greater order. Indeed, the solutions to the deterministic equation $\dot{x}=f(x)$ with $f(x)=ax^\beta+ b x^{\beta-r}$, with $b\neq 0$ and $r=(\alpha+\beta-1)/\alpha$, has to satisfy $x(t)= c_1 t^{1/(1-\beta)} + c_2 t^{1/\alpha}(1+o(1))$, for a non-zero constant $c_2$.

The functional nature of the result in Theorem~\ref{th:main1} is
also important: since
$\gamma+\rho/\alpha=1/\alpha$, the random variables
$t^\gamma B_\alpha(t^\rho)$ and $B_\alpha(t)$ have the same law for
each fixed $t$, but their temporal dependence is generally different
when $\beta\neq0$. Thus a marginal limit would not reveal how the
nonlinear drift amplifies or attenuates fluctuations according to the
time at which they occur.    
    
 An analogous phenomenon, motivating Theorem~\ref{th:main1}, in the context of reflected Brownian motion (RBM) is documented  in~\cite{warwick194470}, where a superdiffusive strong law for RBM holds if a certain parameter takes values in the interval $(-1,1)$, while the CLT is valid only on the sub-interval $(-1/3,1)$. The threshold at $-1/(1+2)$ is linked to the fact that the noise in the RBM model is Brownian (i.e. diffusive), while in SDE~\eqref{eq:sde} the noise scales as $1/\alpha$, inducing a critical threshold for the CLT at $1/(1+\alpha)$. We stress here that, unlike Theorem~\ref{th:main1} above, the CLT in~\cite{warwick194470} holds only marginally (if true, its extension to a functional CLT  requires a new idea).

The problem of establishing a stable CLT for the strong law in~\eqref{eq:sde}-\eqref{eq:asympt_power} has a number of natural generalisations, where the  behaviour described by Theorem~\ref{th:main1} is likely to persist: (a) driving the SDE in~\eqref{eq:sde} by a L\'evy process in the long-time stable domain of attraction and/or (b) allowing for multiplicative noise with, say, globally bounded and regular dispersion coefficient; (c) considering arbitrary dimension. Any combination of generalisations (a)-(c) also appears feasible. However, we expect that the analysis of any of those directions would give rise to many technical difficulties not encountered in the present paper.

We conclude the introduction by a brief discussion of assumption~\eqref{eq:assumption} and a description of the ideas behind the proof of our main result Theorem~\ref{th:main1}. 

\smallskip

\noindent \textbf{A discussion of assumption~\eqref{eq:assumption}.}
The almost sure limit in~\eqref{eq:assumption} holds under the following general assumption:
there exists $r\in\mathbb{R}$ such that  $b_r:=\inf_{x\in(-
\infty,r)} f(x)$ is non-negative  if $\alpha>1$ or greater than $-\infty$  if $\alpha<1$ and  positive jumps in $B_\alpha$ are not absent, i.e. $c_+>0$.

Indeed, under this assumption on $f$,
it is well known from the theory of   
L\'evy processes~\cite[Prop.~37.10 \& Rem.~37.14]{Sato} that $\limsup_{t\to\infty}\left(X(0)+t b_r + B_\alpha(t)\right)=\infty$ a.s., making 
the stopping time $\tau_r:=\inf\{t\geq0:X(t)\geq r\}$
finite a.s. (since, for $t\in[0,\tau_r]$, we have  
$X(t)\geq X(0)+t b_r + B_\alpha(t)$).  Following the reasoning in~\cite[Lem.~A.1]{BM}, note that for any large $R>0$, there exists a non-random time $t_R>0$ such that  
 $\mathbb{P}_x(X(t_R)>R)>0$ for all $x\in[r,R)$ since $\sup_{y\in[r-1,R+1]}|f(y)|<\infty$ and $B_\alpha$ has full support.  As $X$ is Feller continuous~\cite{MR3803683}, its semigroup is lower semicontinuous implying $\epsilon:=\inf_{x\in[r,R]}\mathbb{P}_x(X(t_R)>R)>0$.
Thus, every time $X$ arrives in the interval $[r,\infty)$, which occurs infinitely many times, with probability at least $\epsilon>0$,  $t_R$ units of time later it will be in $[R,\infty)$. The strong Markov property at these entrance times implies $\mathbb{P}(\sup_{t\geq0} X(t)>R)=1$. Since $R>0$ was arbitrary, we obtain $\mathbb{P}(\sup_{t\geq0} X(t)=\infty)=1$,  implying by~\cite[Thm 2.2]{pilipenko2021small} the limit $|X(t)|\to\infty$ as $t\to\infty$ and hence
also~\eqref{eq:assumption}. 

Note also that, since $\beta<1$, the growth of $f$ is sublinear making the process $X$  not  exhibit explosions  (at large positive values, $X$ is bounded above by a stable-driven OU process, which is non-explosive.)

\smallskip

\noindent \textbf{Strategy of the proof of Theorem~\ref{th:main1}.}
The main idea is to straighten the nonlinear drift by introducing the change of variables
$g(x)=\int_{x_0}^{x}\di z/f(z)$, implying
\(g'(x)f(x)=1\) for \(x\geq x_0\). Since \(X(t)\to+\infty\) almost surely, the deterministic part of the dynamics of the process \(g(X(t))\) is eventually simply \(t\). Itô's formula for semimartingales with jumps, given in~\eqref{eq:ito} below, then decomposes \(g(X(t))-t\) into a stochastic integral
\begin{equation}
\label{eq:stochastic_integral_Stable}
\int_0^t g'(X(s-))\,\di B_\alpha(s)
\end{equation}
and a nonlinear correction arising from the Poisson point process of jumps of the stable process \(B_\alpha\). The latter turns out to be negligible on the fluctuation scale: Lemma~\ref{le:estimeg''} gives a second-order estimate for the jump correction, while Lemmas~\ref{le:deltasmall1} and~\ref{le:deltasmall2}, together with Corollary~\ref{cor:bounds_jumps}, show that the jumps are eventually small relative to \(X(s-)\). These estimates are combined in Proposition~\ref{prop:error}.

It remains to identify the scaling limit of the stochastic integral in~\eqref{eq:stochastic_integral_Stable}. The law of large numbers \eqref{eq:asympt_power} implies
\[
g'(X(s-))\sim a^{-1}c_1^{-\beta}s^{-\gamma}\quad\text{as $s\to\infty$.}
\]
Lemma~\ref{coro:scale} identifies the weighted stable integral
\(\int_0^t s^{-\gamma}\,\di B_\alpha(s)\) with the time-changed stable process
\(B_\alpha(t^\rho/\rho)\), and Proposition~\ref{prop:fluct} proves that replacing
\(g'(X(s-))\) by its deterministic asymptotic equivalent does not change the limit. Propositions~\ref{prop:error} and~\ref{prop:fluct} thus imply, via Corollary~\ref{coro:cvg}, the functional limit theorem for \(g(X(\lambda t))-\lambda t\).

Finally, the second-order assumption on \(f\) gives
\[
g(x)=\frac{x^{1-\beta}}{a(1-\beta)}
     +o\bigl(x^{1-\beta-r}\bigr)+O(1)\quad\text{as $x\to\infty$.}
\]
The choice \(r=(\alpha+\beta-1)/\alpha\), together with \eqref{eq:asympt_power}, makes the error negligible on the fluctuation scale; see \eqref{eq:tempgasymp}. Thus the limit for \(g(X)\) first transfers to \(X^{1-\beta}\). A first-order Taylor expansion of the inverse power map about the deterministic trajectory \(c_1(\lambda t)^{1/(1-\beta)}\) gives the limit stated in Theorem~\ref{th:main1} (see~\eqref{eq:asympX} and the subsequent argument below). The final mean-value estimate deals with the behaviour at \(t=0\), directly when \(\beta\geq0\) and in the weighted formulation when \(\beta<0\).

\section{Proof of the main result}

\subsection{Notation}
 We recall and define 
\[ \gamma\coloneqq \frac{\beta}{1-\beta}, \qquad 
\rho\coloneqq 1-\frac{\alpha \beta}{1-\beta}=1-\alpha \gamma,
  \qquad c_1= (a (1-\beta))^{\frac{1}{1-\beta}},\qquad  c_{\pm}(u)=c_+\mathbbm{1}_{u>0}+c_-\mathbbm{1}_{u<0} .
\]
For $\beta<1$ and $\alpha>0$, which we assume throughout the paper, 
the three conditions $\beta<\frac{1}{1+\alpha}$,  $\rho>0$ and $\alpha \gamma<1$ are equivalent to each other. Moreover, the conditions $\beta>1-\alpha$ and $\rho<\alpha $ are equivalent to each other:

\[
 1-\frac{\alpha \beta}{1-\beta} <\alpha \iff 
(1-\alpha)(1-\beta)< \alpha \beta
\iff \beta>1-\alpha.
\]
We will say that a family of processes converges weakly if the convergence holds in distribution in the space of càdlàg functions endowed with the topology of local uniform convergence on $[0,\infty)$.

We fix some $x_0> 0$ such that $f(x)>0$ for all $x\geq x_0$. For $x\geq x_0$, we set \[g(x)\coloneqq \int_{x_0}^x \frac{1}{f(z)}\mathrm{d} z, \]
and we extend $g$ into a $\mathcal{C}^2$ function defined on the whole of $\mathbb{R}$, equal to $0$ on $(-\infty,0]$ (this extension is still called $g$). For $x\geq x_0$, it holds $g'(x)f(x)=1$, and we easily deduce from the first order asymptotic $f(x)\sim a x^{\beta}$ that \[g(x)\sim \frac{x^{1-\beta} }{(1-\beta)a}\text{ as $x\to \infty$,}
\qquad \text{and thus by~\eqref{eq:assumption}} \qquad  
g(X(t))\sim t \text{ as $t\to \infty$.}
\]

We let \[N\coloneqq\{ (s,u)\in \mathbb{R}_+\times (\mathbb{R}\setminus \{0\}): \Delta B_\alpha(s)=u\}\]
be the collection of jumps of $B_\alpha$, and we also write $N$ for the counting measure on the set $N$. 

The constants that appear in the paper depend on $f$ and $\beta$. Additional parameters on which they may depend are written as a subscript (e.g. $C_T$ is a constant that depends on $T$), and we use a subscript $\omega$ for random constants.  

\subsection{Estimation of $g(X(t))$}
The goal of this section is to prove that the family of processes $t\mapsto \lambda^{-1/\alpha+\beta/(1-\beta)}  \Big(g(X({t\lambda}))-t\lambda\Big)$ converges in distribution, as $\lambda\to \infty$, to the process   
\[ 
 (a^{-1}c_1^{-\beta}    \rho^{-1/\alpha} B_\alpha( t^\rho))_{t\in[0,\infty)}.
\]
More precisely, by Itô's formula and SDE~\eqref{eq:sde},
\begin{align}
\label{eq:ito}
   g(X(t))= g(X(0)) & + \int_{0}^t g'(X(s)) f(X(s)) \di s +\int_{0}^t g'(X(s-)) \di  B_\alpha(s)\\
&   + 	\int_{0}^t\int_\mathbb{R}\left( g(X(s-)+u) - g(X(s-)) -g'(X(s-))u\right) N(\di s, \di u). \nonumber 
\end{align}
The first integral in \eqref{eq:ito} is eventually affine, and we will show that after rescaling, the second one converges in distribution to $t\mapsto B_\alpha( t^{\rho})$ while the third is of a smaller order. 

\subsubsection{General properties} 
We start with a few elementary facts that we will use through the proof.
\begin{lemma}
    \label{coro:scale}
    Assume $\rho>0$. Then, the process
    \[Y: t\mapsto \int_0^t s ^{ -\frac{\beta}{1-\beta}}\di B_\alpha(s)
    \]
    is well-defined for all $t\in[0,\infty)$, and equal in distribution to the process $t\mapsto B_\alpha( t^{\rho}/\rho )$.
\end{lemma}
\begin{proof}
    This amounts to showing that the process $u\mapsto Y( \rho^{1/\rho} u^{1/\rho}) $ is an $\alpha$-stable process, which is given by an application of \cite[Theorem 4.1]{Kallenberg:1992} (or \cite[Theorem 3.1]{Rosinski+Woyczynski:1986} in the case $B_\alpha$ is a symmetric stable process) to the function $F:s\mapsto s^{-\frac{\beta}{1-\beta}}$ (which satisfies $\int_0^\delta F(s)^\alpha\di s<\infty$ for any $\delta>0$, since $\rho=1-\alpha\beta/(1-\beta)>0$).
\end{proof}

\begin{lemma} 
    \label{le:estimeg''}
    There exists a finite constant $C$ such that for all $x\geq 2x_0$ and all $u\in [-x/2,x/2]$, 
    \[ 
    |g(x+u)-g(x)- u g'(x)  | \leq  C  u^2 x^{-1-\beta}.
    \]
    \end{lemma} 
    \begin{proof} 
    Since $g''=-(\ln f)'/f$, and since 
    $(\ln f)'(x)=O(1/x)$ and $f(x)\sim a x^\beta$, we deduce $g''(x)=O(x^{-1-\beta})$, as $x\to \infty$. 
    Let $C$ be such that for all $x\geq x_0$, 
    $ 
    |g''(x)|\leq Cx^{-1-\beta} $. 
    Thus, for all $x\geq 2 x_0$, and for $u\in[-x/2,x/2]$, 
    \[\forall z\geq x-|u|\geq x/2,  |g''(z)|\leq C z^{-1-\beta}\leq 2^{1+\beta} C x^{-1-\beta} , \]
    hence
    \begin{align*}
    |g(x+u)-g(x)- u g'(x)  | & = \left|
    \int_{x}^{x+u}  (g'(z) -g'(x)) \mathrm{d}  z\right| 
    \leq \frac{u^2}{2} \sup \{ |g''(z)|: z \in [x-|u|, x+|u|]  \}\\
    & \leq \frac{2^{1+\beta} C u^2}{2}x^{-1-\beta}. \qedhere 
    \end{align*}
\end{proof} 
\begin{lemma}
\label{le:deltasmall1}
    Almost surely, 
    \[
    \forall \delta>0, \ \lim_{s\to\infty, (s,u)\in N} \frac{ |u| }{ s^{\frac{1}{\alpha}+\delta}}=0. 
    \]
\end{lemma}
\begin{proof}
    It is well known that $\lim_{s\to\infty} |B_\alpha(s)|/s^{\frac{1}{\alpha}+\delta}=0$ a.s., see for example \cite[Prop. 48.10]{Sato}. Therefore
    $\lim_{s\to\infty}|\Delta B_\alpha(s)|/s^{\frac{1}{\alpha}+\delta}=0$ a.s., which concludes the proof.
\end{proof}
\begin{lemma} 
    \label{le:deltasmall2}
    Assume $\beta>1-\alpha$. Then, there exists a random, almost surely finite, time $\tau$ (which is not a stopping time) such that for all $(s,u)\in N $ with $s\geq \tau$, it holds $2|u|\leq X(s-)$. 
\end{lemma}
\begin{proof}
    Let $\delta\in(0, 1/(1-\beta)-1/\alpha )$. Then, as $s_0 \to \infty$, $\sup_{\substack{(s,u)\in N \\ s\leq s_0}}   |u|\ll  s_{0}^{1/\alpha + \delta }$ by Lemma \ref{le:deltasmall1}. Since 
    $ s_{0}^{1/\alpha + \delta }
    \ll s_{0}^{1/(1-\beta)}\propto X(s_{0}-)$, the result follows.  
    %
    %
    \end{proof}

    It is clear from this proof that the factor of $2$ in the inequality in Lemma~\ref{le:deltasmall2} is arbitrary and can be substituted by any fixed positive scalar. The factor of $2$ is chose to facilitate the application of Lemma~\ref{le:deltasmall1} in the proof of Corollary~\ref{cor:bounds_jumps} below.
    \begin{corollary}\label{cor:bounds_jumps}
    With probability 1, the limit
    \[
    C_\omega\coloneqq \limsup_{s\to\infty, (s,u)\in N} \frac{|g(X(s-)+u) - g(X(s-)) -g'(X(s-))u|}{s^{\frac{-1-\beta}{1-\beta}}u^2} \quad \text{is finite.}
    \]
\end{corollary}
\begin{proof}
    Let $\tau$ be the random time provided by Lemma \ref{le:deltasmall2}, and let $\tau'$ be a random time such that $X(s)\geq x_0$ for all $s\geq \tau'$ (which exists by the assumption that $X(t)\to +\infty$). Then, for all $(s,u)\in N$ such that $s
    \geq \max(\tau,\tau'+1)$, it holds that $X(s-)\geq x_0$ and that $|u|\leq X(s-)/2$ (the random time $\tau'+1$ guarantees that the left limit $X(s-)$ is greater than $x_0$). Thus, we can apply Lemma \ref{le:estimeg''}, and we deduce that 
    \[
    |g(X(s-)+u)-g(X(s-))- u g'(X(s-))  | \leq  C  u^2 X(s-)^{-1-\beta}\sim 
    C'  u^2 s^{\frac{-1-\beta}{1-\beta}}
    \]
    for some deterministic constants $C,C'$. 
\end{proof}

\subsubsection{Estimating $g(X)$ with $Y$}
Recall that the process $Y$ was defined in Lemma~\ref{coro:scale}.
\begin{proposition}
    \label{prop:error}
    Assume $1-\alpha<\beta<\frac{1}{\alpha+1}$. Then the family of stochastic processes indexed by $\lambda>0$,
    \[t \mapsto  \lambda^{-\frac{\rho}{\alpha} } \int_0^{\lambda t }\int_\mathbb{R}\left( g(X(s-)+u) - g(X(s-)) -g'(X(s-))u\right) N(\di s, \di u), \]
    converges in probability  to $0$ locally uniformly as $\lambda\to\infty$.
\end{proposition}
\begin{proof} 
    We need to prove that for all $T<\infty$, 
    \[      \lambda^{-\frac{\rho}{\alpha} }  \sup_{t\in[0,T]}\left|\int_0^{\lambda t}\int_\mathbb{R}\left( g(X(s-)+u) - g(X(s-)) -g'(X(s-))u\right) N(\di s, \di u)\right|\underset{\lambda \to \infty}{\overset{\mathbb{P}}{\longrightarrow}}0,
    \]
    for which it suffices to prove that 
    \[      \lambda^{-\frac{\rho}{\alpha} } \int_0^{\lambda T}\int_\mathbb{R}\left| g(X(s-)+u) - g(X(s-)) -g'(X(s-))u\right| N(\di s, \di u)\underset{\lambda \to \infty}{\overset{\mathbb{P}}{\longrightarrow}}0.
    \]
    
    Let $C_\omega$ be the random variable from Corollary \ref{cor:bounds_jumps}, which is almost surely finite. Let also $\tau_C$ be an almost surely finite random time (which is not a stopping time) such that for all $(s,u)\in N$ with $s\geq \tau_C$, it holds  
    $|g(X(s-)+u) - g(X(s-)) -g'(X(s-))u|\leq (C_\omega+1)s^{\frac{-1-\beta}{1-\beta}}u^2$.
    Fix $\delta>0$ 
    and let $\tau_\delta$ be such that for all $(s,u)\in N$ with $s\geq \tau_\delta$, it holds $|u| /s^{1/\alpha+\delta}\leq 1$.
    Then $\tau_\delta$ is almost surely finite by Lemma \ref{le:deltasmall1}.  We let $\tau=\max(\tau_C,\tau_\delta,1)$. 
    
    Let $U=\sup\{ |u|: \exists s \leq \tau: (s,u)\in N \}$. Since $\tau$ is almost surely finite, so is $U$. 
    Since furthermore $\|X\|_{\infty,[0,\tau]}<\infty$ almost surely and $g$ is $\mathcal{C}^2$, we deduce that almost surely, the sum 
    \[
    \int_0^{\tau }\int_\mathbb{R}\left| g(X(s-)+u) - g(X(s-)) -g'(X(s-))u\right| N(\di s, \di u)\]
    is finite.
    Hence, 
    \[  
    \lambda^{-\frac{\rho}{\alpha} }
    \int_0^{\tau }\int_\mathbb{R}\left| g(X(s-)+u) - g(X(s-)) -g'(X(s-))u\right| N(\di s, \di u)\underset{\lambda\to \infty}{\longrightarrow} 0 \qquad \text{almost surely}.
    \]
    We now consider 
    \[ 
    I_{\lambda,T}\coloneqq \lambda^{-\frac{\rho}{\alpha} }
    \int_{\tau\wedge \lambda T}^{\lambda T}   \int_\mathbb{R}\left| g(X(s-)+u) - g(X(s-)) -g'(X(s-))u\right| N(\di s, \di u).
    \]
    Since $\tau\geq \tau_C$, for all $(s,u)\in N$ with $s\geq \tau$, 
    \[\left| g(X(s-)+u) - g(X(s-)) -g'(X(s-))u\right| 
    \leq (C_\omega+1) s^{\frac{-1-\beta}{1-\beta}}u^2.
    \]
    Furthermore, since $\tau\geq \tau_\delta$, for all $(s,u)\in N$ with $s\geq \tau$, it holds  $\mathbbm{1}_{|u|\leq s^{\frac{1}{\alpha}+\delta }}=1$. 
    Thus, 
    \[ 
    I_{\lambda,T} 
    \leq (C_\omega+1)
    \lambda^{-\frac{\rho}{\alpha} }
    \int_{1\wedge \lambda T}^{\lambda T}   \int_\mathbb{R}   s^{\frac{-1-\beta}{1-\beta}}u^2 \mathbbm{1}_{|u|\leq s^{\frac{1}{\alpha}+\delta }  } N(\di s, \di u).
    \]
    For $\lambda\geq 1/T$, 
    \begin{align*}
    \mathbb{E}\Big[
    \lambda^{-\frac{\rho}{\alpha} }
    \int_{1}^{\lambda T}   \int_\mathbb{R}   s^{\frac{-1-\beta}{1-\beta}}u^2 \mathbbm{1}_{|u|\leq s^{\frac{1}{\alpha}+\delta }  } N(\di s, \di u)
    \Big]
    &=
    \lambda^{-\frac{\rho}{\alpha} }
    \int_{1}^{\lambda T}  s^{\frac{-1-\beta}{1-\beta}}  \int_{-s^{\frac{1}{\alpha}+\delta }}^{s^{\frac{1}{\alpha}+\delta } }   u^2 \frac{c_{\pm}(u) \di u \ \di s}{u^{1+\alpha}}\\
    &= \frac{(c_++c_-) \lambda^{-\frac{\rho}{\alpha} } }{2-\alpha}
    \int_{1}^{\lambda T}  s^{\frac{-1-\beta}{1-\beta} + (2-\alpha)(1/\alpha+\delta)  } \di s\\
    &= \frac{(c_++c_-) \lambda^{-\frac{\rho}{\alpha} } }{2-\alpha}
    \int_{1}^{\lambda T}  s^{  2/\alpha-2/(1-\beta)    + (2-\alpha)\delta  } \di s\\
    &\leq C_{T,\delta}  \lambda^{-\frac{\rho}{\alpha} }    (1+ \lambda^{1+2/\alpha-2/(1-\beta)    + (2-\alpha)\delta     } ). 
    \end{align*}
    In the allowed range for $\beta$, the exponent $1+2/\alpha-2/(1-\beta)$ can be either positive or negative. Recall $-\frac{\rho}{\alpha} =-\frac{1}{\alpha}+  \frac{ \beta}{1-\beta}$ and $\rho>0$. So
    \[
    -\frac{\rho}{\alpha}+  1+\frac2\alpha-\frac{2}{1-\beta}=
-\frac{1}{\alpha}+  \frac{ \beta}{1-\beta}+1+\frac2\alpha-\frac{2}{1-\beta}=\frac{1-\alpha-\beta}{\alpha(1-\beta)}<0.
    \]
    Hence, for any  $0<\delta<-\frac{1-\alpha-\beta}{\alpha(1-\beta)}\cdot\frac{1}{2-\alpha}$, we have 
    \[ \lambda^{-\frac{\rho}{\alpha} }    (1+ \lambda^{1+2/\alpha-2/(1-\beta)    + (2-\alpha)\delta     } )\to0,\quad\lambda\to\infty.\]
   Thus 
    \[
    \mathbb{E}\Big[
    \lambda^{-\frac{1}{\alpha}+  \frac{ \beta}{1-\beta}}
    \int_{1}^{\lambda T}   \int_\mathbb{R}   s^{\frac{-1-\beta}{1-\beta}}u^2 \mathbbm{1}_{|u|\leq s^{\frac{1}{\alpha}+\delta }  } N(\di s, \di u)
    \Big]\underset{\lambda\to \infty}\longrightarrow 0.\]
    In particular, 
    \[
    \lambda^{-\frac{1}{\alpha}+  \frac{ \beta}{1-\beta}}
    \int_{1}^{\lambda T}   \int_\mathbb{R}   s^{\frac{-1-\beta}{1-\beta}}u^2 \mathbbm{1}_{|u|\leq s^{\frac{1}{\alpha}+\delta }  } N(\di s, \di u)\underset{\lambda\to \infty}{\overset{\mathbb{P}}{\longrightarrow}} 0,
    \]
    and it follows $I_{\lambda,T}$ goes to $0$ in probability, which concludes the proof of the proposition. 
\end{proof}

\begin{proposition}
    \label{prop:fluct}
    Assume $1-\alpha<\beta<\frac{1}{\alpha+1}$. Then  the family of stochastic processes indexed by $\lambda>0$,
    \[t \mapsto  \lambda^{-\frac{\rho}{\alpha}} \int_0^{\lambda t } g'(X(s-))\di B_\alpha(s), \]
    converges in distribution, locally uniformly as $\lambda\to \infty$, 
    toward 
    \[ 
    U: t\mapsto   a^{-1}c_1^{-\beta}  \rho^{-1/\alpha} B_\alpha( t^\rho).
    \]
    Moreover, there exists a family of copies $U_\lambda$, $\lambda>0$, of $U$, defined on the initial probability space, such that 
   $ (\lambda^{-\frac{\rho}{\alpha}} \int_0^{\lambda t } g'(X(s-))\di B_\alpha(s) -U_\lambda(t))$
    converges   locally uniformly to zero in probability as $\lambda\to \infty$.
\end{proposition}
\begin{proof} 
    Recall $\gamma=\beta/(1-\beta)$, $1/\alpha-\gamma=\rho/\alpha>0$, and $\alpha \gamma<1$. 
    By the scale invariance $B_\alpha( \ \cdot \ \lambda^\rho/\rho)   \overset{(d)}= \lambda^{\rho/\alpha} B_\alpha(\cdot)/\rho^{1/\alpha} $ 
    and Lemma~\ref{coro:scale}, 
    we deduce the equalities in distribution $$(Y( \lambda t ))_{t\in[0,\infty)} \overset{(d)}= (B_\alpha  (  (\lambda t)^\rho/\rho ))_{t\in[0,\infty)}
   \overset{(d)}=   \lambda^{1/\alpha-\gamma}  (B_\alpha(t^\rho))_{t\in[0,\infty)}/\rho^{1/\alpha},
    $$ 
    where $Y(t)= \int_0^t s ^{ -\frac{\beta}{1-\beta}}\di B_\alpha(s).$  We set $U_\lambda(t):= \lambda^{-\rho/\alpha}a^{-1}c_1^{-\beta} Y(\lambda t)$.

    Thus, it suffices to show that, locally uniformly in $t$, as $\lambda\to \infty$,
    \[ \lambda^{-\frac{1}{\alpha}+  \gamma} \Big( \int_0^{\lambda t } g'(X(s-))\di B_\alpha(s)
    -        a^{-1}c_1^{-\beta} Y(\lambda t) \Big) 
    \overset{\mathbb{P}}\longrightarrow 0,\]
    i.e. that
    \[ \forall T>0\quad \sup_{t\leq T} |J_{\lambda,t}| \underset{\lambda\to \infty}{ \overset{\mathbb{P}}\longrightarrow} 0, \qquad \text{where} \qquad
    J_{\lambda,t}\coloneqq  \lambda^{-\frac{1}{\alpha}+  \gamma}\int_0^{\lambda t }  \big( g'(X(s-))
    -  a^{-1}c_1^{-\beta} s^{-\gamma}
    \big) \di B_\alpha(s).
    \]
    
    Since $g'(x)=1/f(x)\sim a^{-1} x^{-\beta}$ and almost surely $X(t)\sim c_1 t^{\frac1{1-\beta}}$, 
    \[ 
    g'(X(s-)) s^{\gamma } \underset{s \to \infty}\longrightarrow  a^{-1}c_1^{-\beta}  \quad a.s.
    \]
    Fix $\epsilon>0$ and let $\tau$ be an almost surely finite random time such that for $s\geq \tau$, 
    \[ 
    Z(s)\coloneqq g'(X(s-))s^{\gamma} 
    - a^{-1}c_1^{-\beta}   \in [-\epsilon, \epsilon].
    \]
    Since $\tau$ is almost surely finite, it suffices to show that $\sup_{t\leq T}
    |  J_{\lambda,t} \mathbbm{1}_{\tau \leq \theta}|\overset{\mathbb{P}}\longrightarrow 0$ for all deterministic $\theta\in(0,\infty)$.
    Remark
    \begin{align*}
    |J_{\lambda,t} \mathbbm{1}_{\tau \leq \theta}|
    &\leq 
    \lambda^{-\frac{1}{\alpha}+  \gamma } \Big|\int_0^{\theta }  (g'(X(s-))
    -  a^{-1}c_1^{-\beta}  s^{-\beta/(1+\beta)}) 
    \di B_\alpha(s) \Big|\\
    &+  \lambda^{-\frac{1}{\alpha}+  \gamma}\Big| \int_\theta^{\lambda t } \mathbbm{1}_{ |Z(s)|\leq \epsilon  } Z(s)s^{-\gamma}
    \di B_\alpha(s)\Big|,
    \end{align*}
    which is easily seen by considering separately the events $\tau \leq \theta$ and $\tau > \theta$ (the last integral is well defined, since $Z$ is predictable).
    Thus, it suffices to prove that 
    \begin{equation}
    \label{eq:temp:asfin}
    \text {almost surely, } \qquad \int_0^{\theta}  (g'(X(s-))
    -  a^{-1} c_1^{-\beta} s^{-\gamma}) 
    \di B_\alpha(s) \qquad \text{is well defined and finite;}
    \end{equation} 
    and that there exists $C_T$ such that for all $\epsilon'>0$, there exist  $\lambda_0$ and $\epsilon>0$ such that for all $\lambda\geq \lambda_0$,
    \begin{equation}
    \label{eq:temp:L2small}
    \mathbb{P}\Big(  \sup_{t<T} \Big|  \lambda^{-\frac{1}{\alpha}+  \gamma} \int_\theta^{\lambda t } \mathbbm{1}_{ |Z(s)|\leq \epsilon  } Z(s)s^{-\gamma}
    \di B_\alpha(s)\Big| \geq  C_T \epsilon'\Big) \leq \epsilon'.
    \end{equation}
    Remark although we used the time $\tau$ which is not a stopping time as an intermediate, we are left with integrals of adapted processes in the end. 
    Since $g'\circ X$ is an adapted process which is almost surely bounded on $[0,\theta]$, the integral 
    \[ 
    \int_0^{\theta}  g'(X(s-))\di B_\alpha(s)
    \]is well-defined and almost surely finite.
    Furthermore 
    $
    \int_0^{\theta}  s^{-\gamma} \di B_\alpha(s)= Y(\theta)$,
    which we have already seen is almost surely finite. Thus \eqref{eq:temp:asfin} holds. 

    To prove \eqref{eq:temp:L2small}, let $B'(u)\coloneqq \lambda^{-\frac{1}{\alpha}} B_\alpha(\lambda u)$, which is equal to $B_\alpha$ in distribution. With the change of variable $s'=\lambda^{-1}s$, we get  
    \begin{align*}
    \lambda^{-\frac{1}{\alpha}+\gamma} \int_\theta^{\lambda t } \mathbbm{1}_{ |Z(s)|\leq \epsilon  } Z(s)s^{-\gamma}
    \di B_\alpha(s)
    &=   \int_{\lambda^{-1}\theta}^{t } \mathbbm{1}_{|Z(\lambda s)|\leq \epsilon} Z(\lambda s) s^{-\gamma} \di B'(s).
    \end{align*}

  We  treat the cases $\alpha<1$ and $\alpha>1$ separately. For \(\alpha<1\), the stable process $B_\alpha$ has finite variation, allowing a pathwise bound on the error in terms of its total variation. For \(\alpha>1\), the paths have infinite variation, so we instead use compensated jump integrals, controlling small jumps by second-moment estimates and large jumps by first-moment estimates.
  
\noindent\underline{Case $\alpha<1$}. In this case, we can decompose $B'$ as $B'=B^+-B^-$, where \[B^+(t)= \int_{[0,t]\times [0,\infty)}  u N(\di s , \di u), \qquad 
B^-(t)= \int_{[0,t]\times (-\infty, 0]}  |u| N(\di s , \di u).\]
Since $\alpha<1$, the processes $B^+$ and $B^-$ are of finite variation.  Thus, 
\begin{align*}
\Big|\int_{\lambda^{-1}\theta}^{ t } \mathbbm{1}_{ |Z(\lambda s)|\leq \epsilon  } Z(\lambda s)s^{-\gamma}
    \di B'(s)\Big| 
& \leq \epsilon \int_0^{ t} s^{-\gamma}
    (\di B^+(s)+\di B^-(s) ) = \epsilon (Y^+(t)+Y^-(t)), 
\end{align*}
where $Y^\pm$ are the process defined as $Y$ but for the stable processes $B^\pm$ instead of $B_\alpha$.  
  In particular, it follows from Lemma \ref{coro:scale} that these are well-defined càdlàg and almost surely finite processes, so that 
  \[     \mathbb{P}\Big(  \sup_{t<T} \Big|  \lambda^{-\frac{1}{\alpha}+  \gamma} \int_\theta^{\lambda t } \mathbbm{1}_{ |Z(s)|\leq \epsilon  } Z(s)s^{-\gamma}
    \di B_\alpha(s)\Big| \geq  C_T \epsilon'\Big)
    \leq \mathbb{P}(   \sup_{t<T} (Y^+(t)+Y^-(t) )\geq C_T \epsilon'/\epsilon   ),
  \]
which can be made arbitrarily small by taking $\epsilon$ small enough, which concludes in the case $\alpha<1$. 

\underline{Case $\alpha>1$}.
In this case we  decompose $B'(u)=\lambda^{-\frac{1}{\alpha}} B_\alpha(\lambda u)$ into summands as follows,
    \begin{align*}
    B'(t)
    &=     \int_0^t \int_{\mathbb{R} }u \mathbbm{1}_{|u|\leq s^\gamma} \tilde{N}(\di s, \di u)
+ \int_0^t  \int_{\mathbb{R} } u \mathbbm{1}_{|u|>s^\gamma }  \tilde{N}(\di s, \di u)
    \eqcolon B'_1(t)+B'_2(t),
    \end{align*}
    where $\tilde{N}$ is the compensated measure, $\tilde{N}(\di s, \di u)=N(\di s, \di u) - \frac{c_{\pm}(u)  }{|u|^{1+\alpha}}\di s \di u $, and where the integrals are understood in the classical sense for compensated Poisson measures (see e.g.~\cite[Ch. 4]{Sato}). 
    
  On the one hand, by Itô isometry applied to the process $B'_1$, 
    \begin{align*}
    \mathbb{E}\Big[ \Big( \int_{\lambda^{-1}\theta}^{T}  \mathbbm{1}_{ |Z(\lambda s)|\leq \epsilon  } Z(\lambda s)s^{-\gamma}
    \di B'_1(s) \Big)^2 \Big]
    &=
    \int_{\lambda^{-1}\theta}^{T} \int_{-s^\gamma}^{s^\gamma} \mathbb{E}[Z(\lambda s)^2\mathbbm{1}_{|Z(\lambda s)|\leq \epsilon}] s^{-2\gamma} u^2 \frac{c_\pm(u) \di u \di s }{|u|^{1+\alpha}} \\
        &\leq
 \epsilon^2   \int_{0}^{T} \int_{-s^\gamma}^{s^\gamma}  s^{-2\gamma} u^2 \frac{c_\pm(u) \di u \di s }{|u|^{1+\alpha}} \\
    &= \epsilon^2 \frac{c_++c_-}{2-\alpha}  \int_0^T s^{-\alpha \gamma} \di s = C_1 \epsilon^2.
    \end{align*}

On the other hand,  
    \begin{align*}
    \mathbb{E}\Big[ \Big| \int_{\lambda^{-1}\theta}^{T}  \mathbbm{1}_{ |Z(\lambda s)|\leq \epsilon  } Z(\lambda s) s^{-\gamma}
    \di B'_2(s)  \Big| \Big]
    & \leq 
    (c_++c_-) \epsilon \int_0^T \int_{s^\gamma}^{\infty} s^{-\gamma}
    u  \frac{\di u \di s }{|u|^{1+\alpha}} \\
    & =
    (c_++c_-)\epsilon \int_0^T  s^{-\gamma} \frac{s^{\gamma(1-\alpha)}  } {\alpha-1}   \di s\\
    &= \frac{(c_++c_-) \epsilon }{(\alpha-1)( 1-\alpha \gamma  ) } T^{1-\gamma \alpha}=C_2\epsilon. 
    \end{align*}
Remark both $C_1$ and $C_2$ are finite because  $\alpha>1$ and $\alpha \gamma<1$.
    

By triangle inequality we obtain
    \begin{align*}
    \mathbb{P}\Big(  \sup_{t\leq T} \Big|  \lambda^{-\frac{1}{\alpha}+  \gamma} \int_\theta^{\lambda t }\mathbbm{1}_{ |Z(s)|\leq \epsilon  } Z(s)s^{-\gamma} \di B_\alpha(s)\Big| \geq  & \epsilon'\Big)  
    \leq \\
    \mathbb{P}\Big( \sup_{t\leq T} \Big| \int_{\lambda^{-1}\theta}^{T}  \mathbbm{1}_{ |Z(\lambda s)|\leq \epsilon  } Z(\lambda s) s^{-\gamma}\di B'_1(s)  \Big|  \geq \frac{\epsilon'}{2}\Big)
    &+\mathbb{P}\Big(  \sup_{t\leq T} \Big| \int_{\lambda^{-1}\theta}^{T}  \mathbbm{1}_{ |Z(\lambda s)|\leq \epsilon  } Z(\lambda s) s^{-\gamma} \di B'_2(s)  \Big|  \geq \frac{\epsilon'}{2}\Big).
    \end{align*}
Using Doob's martingale inequality and Markov's inequality, we thus get 
    \begin{align*}
   & \mathbb{P}\Big(  \sup_{t\leq T} \Big|  \lambda^{-\frac{1}{\alpha}+  \gamma} \int_\theta^{\lambda t }\mathbbm{1}_{ |Z(s)|\leq \epsilon  } Z(s)s^{-\gamma} \di B_\alpha(s)\Big| \geq   \epsilon'\Big)  
    \\
    & \leq 4\epsilon'^{-2}		\mathbb{E}\Big[ \Big( \int_{\lambda^{-1}\theta}^{T }    \mathbbm{1}_{ |Z(\lambda s)|\leq \epsilon  } Z(\lambda s)s^{-\gamma}
    \di B'_1(s)\Big)^2 \Big] +2    \epsilon'^{-1}		\mathbb{E}\Big[ \Big| \int_{\lambda^{-1}\theta}^{T} \mathbbm{1}_{ |Z(\lambda s)|\leq \epsilon  } Z(\lambda s) s^{-\gamma}
    \di B'_2(s)\Big| \Big]   \\
    &	\leq 4 C_1(\epsilon/\epsilon')^2+2C_2 \epsilon/\epsilon',
    \end{align*}
    which is smaller than $\epsilon'$ provided $\epsilon$ is sufficiently small,  concluding the proof. 
\end{proof}

\begin{corollary}
\label{coro:cvg}
    As $\lambda \to \infty$, 
    the process 
    $t\mapsto \lambda^{-\frac{\rho}{\alpha}} (g(X(\lambda t))-\lambda t)$ converges in distribution, locally uniformly, toward  
    \[ 
    U: t\mapsto    a^{-1}c_1^{-\beta}  \rho^{-1/\alpha} B_\alpha( t^\rho).
    \]
   Moreover, there exists a family of copies $U_\lambda$, $\lambda>0$, of $U$, defined on the initial probability space, such that 
   $ (\lambda^{-\frac{\rho}{\alpha}} (g(X(\lambda t))-\lambda t) -U_\lambda(t))$
    converges   locally uniformly in probability as $\lambda\to \infty$  
    toward  0.
\end{corollary}
\begin{proof}
    Let $\tau\coloneqq \inf\{ t: \forall t'\geq t, X(t')\geq x_0+1\}$, which is almost surely finite by assumption. 
    Let $t_0>0$ deterministic. 
Almost surely on the event $\tau\leq t_0$, we have $\int_{t_0}^t g'(X(s)) f(X(s)) \di t=t-t_0$. By \eqref{eq:ito}, we get 
\begin{align*}
    g(X(t))    = t-t_0 & +g(X(t_0)) +\int_{t_0}^t g'(X(s-)) \di  B_\alpha(t) \\
   & + 	\int_{t_0}^t\int_\mathbb{R}\left( g(X(s-)+u) - g(X(s-)) -g'(X(s-))u\right) N(\di s, \di u).
\end{align*}
Since this holds for $t_0$ arbitrary, we deduce that almost surely, for all $t\geq \tau$,
    \[ 
    g(X(t))=C+ t
    +\int_{0}^t g'(X(s)) \di  B_\alpha(s)
    + 	\int_{0}^t\int_\mathbb{R}\left( g(X(s-)+u) - g(X(s-)) -g'(X(s-))u\right) N(\di s, \di u),
    \]
    where 
    \begin{align*}
    C\coloneqq-t_0&+g(X(t_0))+\int_0^{t_0} g'(X(s-)) \di  B_\alpha(s) \\
   &+ 	\int_0^{t_0}\int_\mathbb{R}\left( g(X(s-)+u) - g(X(s-)) -g'(X(s-))u\right) N(\di s, \di u)
   \end{align*}
   is an almost surely finite random variable. 
    We conclude by applying Propositions \ref{prop:error} and \ref{prop:fluct}. 
\end{proof}

\subsubsection{Conclusion of the proof of Theorem \ref{th:main1}}
We can now prove our main result.
Since $X(s)\to+\infty$ a.s. and $X$ is càdlàg, its negative part
is a.s. bounded. Replacing $X$ by $X\vee0$ therefore changes
$\lambda^{-1/\alpha}X(\lambda t)$ by $o(1)$ uniformly in $t\geq0$
a.s., while leaving $g(X)$ unchanged. We make this replacement
throughout the remainder of this subsection, retaining the
notation $X$.
    The condition on $f$ ensures that 
    $
    g(x)= \frac{x^{1-\beta}}{a (1-\beta)} +o(x^{1-\beta- r })+O(1)$ as $x\to +\infty$.
    Thus, as $x\to \infty$ and $\lambda\to \infty$, 
    \[
    \lambda^{-\rho/\alpha} ( g(x)- \frac{ x^{1-\beta}}{a (1-\beta)} ) =o( \lambda^{-\rho/\alpha}  x^{1-\beta- r })+o(1).\]
    The value of $r$ is such that, using the asymptotic \eqref{eq:asympt_power}, we get for all $T<\infty$, almost surely, 
    \begin{equation}
    \label{eq:tempgasymp}
    \sup_{t\in [0,T]} |
    \lambda^{-\rho/\alpha} (g(X(\lambda t))-\frac{X(\lambda t)^{1-\beta}}{a (1-\beta)}) |  
    \underset{\lambda \to \infty}=o(
    \sup_{t\in [0,T]} | \lambda^{-\rho/\alpha}  X(\lambda t)^{1-\beta- r } |)+o(1) \underset{\lambda \to \infty}\longrightarrow 0. 
    \end{equation}
    By Corollary \ref{coro:cvg}, there exist copies $U_\lambda$ of the process $U$  such that 
    $t\mapsto \Big(\lambda^{-\frac{\rho}{\alpha}} (g(X(\lambda t))-\lambda t)-U_\lambda(t)\Big)$ converges  locally uniformly in probability toward 0 as $\lambda \to\infty$. 
    Using also \eqref{eq:tempgasymp}, 
    we deduce that the processes $t\mapsto \Big(\lambda^{-\rho/\alpha} \big( \frac{X(\lambda t)^{1-\beta}  }{a (1-\beta)} -\lambda t \big)-U_\lambda(t)\Big)$ converges     locally uniformly  in probability toward  0.
    We now use the notation $A_{\lambda,t}=B_{\lambda,t}+ o(\lambda^r)$ as a shortcut notation for  
    ``the process 
    $
    t\mapsto \lambda^{-r} (A_{\lambda,t}-B_{\lambda,t})$ converges toward $0$    in probability   locally uniformly.'' 
Hence
    \begin{equation} 
    \label{eq:asympX} 
    X(\lambda t)^{1-\beta}  =a (1-\beta)\lambda t+ a (1-\beta) U_\lambda(t)\lambda^{\rho/\alpha}+o(\lambda^{\rho/\alpha})
    =a (1-\beta) \lambda t (1+  t^{-1} (U_\lambda(t)+o(1))    \lambda^{\rho/\alpha-1}).
    \end{equation}  
     It follows that, uniformly over $t\in [\delta,C]$ for arbitrary $0<\delta<C<\infty$, 
    \begin{align*} X(\lambda t)
    & =(a (1-\beta) \lambda t)^{1/(1-\beta)}  (1+  (1-\beta)^{-1} t^{-1} (U_\lambda(t)+o(1) )\lambda^{\rho/\alpha-1})
    \\
    &= (a (1-\beta)\lambda t)^{1/(1-\beta)} + (a (1-\beta)\lambda t)^{1/(1-\beta)} (1-\beta)^{-1} t^{-1} (U_\lambda(t)+o(1))    \lambda^{\rho/\alpha-1},
    \end{align*}
    i.e. 
    \begin{align*}  
    X(\lambda t)  - (a (1-\beta) \lambda t)^{1/(1-\beta)}& =(a (1-\beta))^{1/(1-\beta)} (1-\beta)^{-1} t^{\frac{\beta}{1-\beta}} (U_\lambda(t)+o(1) )    \lambda^{1/(1-\beta)+ \rho/\alpha-1}.
     \end{align*}
So, $(X^{(\lambda)}(t)- a c_1^\beta \rho^{1/\alpha} U_\lambda(t)t^{\gamma})_{ t\geq0}$ converges to $0$ in probability  uniformly on $[\delta,C]. $ Since  $(a c_1^\beta \rho^{1/\alpha} U_\lambda(t) )_{ t\geq0}\overset{(d)}{=} (B_\alpha(t^\rho))_{t\geq 0},$ this completes the proof of the first part of the theorem.

It only remains to extend the convergence at $0$ in the case $\beta\geq0$. Recall \eqref{eq:asympX} and rearrange it  as follows
    \[ a (1-\beta) U_\lambda(t)+o(1) =\lambda^{-\rho/\alpha}\left(X(\lambda t)^{1-\beta} -a (1-\beta)\lambda t\right)= \lambda^{1-\rho/\alpha}\left(\left(\frac{X(\lambda t)}{\lambda^{\frac{1}{1-\beta}}}\right)^{1-\beta} -a (1-\beta) t\right).
    \] 
So
 \begin{align*}
 \frac{X(\lambda t)}{\lambda^{\frac{1}{1-\beta}}}  -(a (1-\beta) t)^{\frac{1}{1-\beta}}
& =
\left(\frac{X(\lambda t)}{\lambda^{\frac{1}{1-\beta}}}\right)^{(1-\beta)\frac{1}{1-\beta}} -\left(a (1-\beta) t\right)^{\frac{1}{1-\beta}} \\
&=
\left(\left(\frac{X(\lambda t)}{\lambda^{\frac{1}{1-\beta}}}\right)^{1-\beta} -a (1-\beta) t\right)    \frac{1}{1-\beta}\theta_\lambda
(a (1-\beta) t)^{\frac{\beta}{1-\beta}}\\
&=\lambda^{-1+\rho/\alpha}(a (1-\beta) U_\lambda(t)+o(1)) \frac{1}{1-\beta}\theta_\lambda
(a (1-\beta) t)^{\frac{\beta}{1-\beta}},
 \end{align*}
    where $\theta_\lambda(a (1-\beta) t)$ is a point between $a (1-\beta) t$ and $\left(\frac{X(\lambda t)}{\lambda^{\frac{1}{1-\beta}}}\right)^{1-\beta}.$ 
    Notice that $\left(\frac{X(\lambda t)}{\lambda^{\frac{1}{1-\beta}}}\right)^{1-\beta} $ 
   converges locally uniformly to $a (1-\beta) t$ as $\lambda\to\infty. $ 
Since $\beta\geq0$, the derivative of
$x\mapsto x^{1/(1-\beta)}$ is continuous on $[0,\infty)$.
The mean-value derivative therefore converges locally uniformly
in probability, including at zero. For $\beta=0$, this derivative
is identically $1$, and we interpret the zero powers below as $1$.
Therefore,
\begin{align*}
&\lambda^{-1/\alpha}
\left[
X(\lambda t)
-\bigl(a(1-\beta)\lambda t\bigr)^{1/(1-\beta)}
\right]
=
aU_\lambda(t)+o(1)
\left[
\bigl(a(1-\beta)t\bigr)^{\beta/(1-\beta)}
+o(1)
\right],
\end{align*}   
which after  simplifying the constants gives 
 \begin{align*}
 \lambda^{-1/\alpha}\left(  {X(\lambda t)}   -(a (1-\beta) \lambda t)^{\frac{1}{1-\beta}}\right)
 &- a  U_\lambda(t) 
(a (1-\beta) t)^{\frac{\beta}{1-\beta}}\\& 
 = U_\lambda(t) o(1)+ 
t^{\frac{\beta}{1-\beta}} o(1)+o(1)=o(1),
\end{align*}
since $\beta\geq 0.$

Thus
$(X^{(\lambda)}(t)- a c_1^\beta \rho^{1/\alpha} U_\lambda(t)t^{\gamma})_{ t\geq0}$ converges to $0$ in probability  uniformly on every $[0,C] $.  
Since 
 $(a c_1^\beta \rho^{1/\alpha}U_\lambda(t) )_{ t\geq0}\overset{(d)}{=} (B_\alpha(t^\rho))_{t\geq 0}$,
this concludes the proof of Theorem~\ref{th:main1}. 

\smallskip

It remains to state and prove the following result announced in Section~\ref{sec:Introduction}, which is analogous to the failure of the CLT given in~\cite[Prop.~6.6]{warwick194470}.

\begin{proposition}
\label{le:counter}
    Assume $\beta>\frac{1}{1+\alpha}$ instead. Then the family of random variables indexed by $\lambda>0$,
    \[ \frac{X(\lambda )-c_1 \lambda^\frac{1}{1-\beta}}{\lambda^{1/\alpha}}\] 
    does \textbf{not} converge in distribution as $\lambda\to+\infty$.
\end{proposition}
\begin{proof}


Let $(\tilde{B}_\alpha,\tilde{X})$ be distributed as a solution $(B_\alpha,X)$ of SDE~\eqref{eq:sde}, 
and coupled to it in such a way that, in an event $E$ of probability $p>0$, $\tilde{X}(2)=X(1)$, and such that for all $t\geq 1$ the increment $\tilde{B}_\alpha(t+1)-\tilde{B}_\alpha(2)$ is taken to be equal to $B_\alpha(t)-B_\alpha(1)$. Thus, on the event $E$, for all $t\geq 1$, $\tilde{X}(t+1)=X(t)$. 

We proceed by contradiction and assume $W(\lambda)\coloneqq \frac{X(\lambda )-c_1 \lambda^\frac{1}{1-\beta}}{\lambda^{1/\alpha}}$ converges in distribution. Since $\tilde{X}$ is distributed as $X$, it then also holds that $\tilde{W}(\lambda)\coloneqq  \frac{\tilde{X}(\lambda )-c_1 \lambda^\frac{1}{1-\beta}}{\lambda^{1/\alpha}}$ converges in distribution. In particular, for all $\epsilon>0$, $\lambda^{-\epsilon} (W(\lambda)-\tilde{W}(\lambda+1))$ converges toward $0$ in probability as $\lambda\to \infty$. 
Let $\epsilon\in(0, \frac{1}{1-\beta}-\frac{1}{\alpha}-1)$, which exists since $\beta>1/(1+\alpha)$. 
In the event $E$, 
\begin{align*} 
\lambda^{-\epsilon} (W(\lambda)-\tilde{W}(\lambda+1))
= 
\lambda^{-\epsilon} 
\frac{-c_1 \lambda^\frac{1}{1-\beta}+ c_1 ( \lambda+1 )^\frac{1}{1-\beta}   }{\lambda^{1/\alpha}}
&=c_1 \lambda^{\frac{1}{1-\beta}-\frac{1}{\alpha}-\epsilon  } (   (1+\lambda^{-1})^\frac{1}{1-\beta}-1)\\
& \sim \frac{c_1}{1-\beta} 
\lambda^{\frac{1}{1-\beta}-\frac{1}{\alpha}-1-\epsilon  }\underset{\lambda\to \infty}\longrightarrow +\infty, 
\end{align*}
which contradicts the fact the same quantity converges toward $0$ and therefore concludes the proof. 
\end{proof}

	\section*{Acknowledgements}
\phantom{ }
\vspace{-1em}

AM was supported in part by EPSRC grants EP/V009478/1 and EP/W006227/1. AP thanks  the Swiss National Science
Foundation for partial support  of the paper (grants No. IZRIZ0\_226875, No. 200020\-\_200400, No. 200020\_192129). IS was supported by the EPSRC grant EP/W006227/1.

The authors thank the Isaac Newton Institute for Mathematical Sciences at the University of Cambridge for support and hospitality during the programme \emph{Stochastic systems for anomalous diffusion}. This work was supported by EPSRC grant EP/Z000580/1.

	\bibliographystyle{plain}
	\bibliography{pilipenko_AMS_2024}

\end{document}